\documentclass[12pt]{amsart}

\usepackage[utf8]{inputenc}

\usepackage[T1]{fontenc}

\usepackage[a4paper, margin=2.7cm]{geometry}

\usepackage{amsmath, amsthm, amsfonts, amssymb}
\usepackage{mathrsfs}           

\usepackage[colorlinks=true,linkcolor=blue,citecolor=blue,urlcolor=blue,breaklinks]{hyperref}

\usepackage{bbm}

\newcommand{\C}{{\mathbb C}}                  

\def\al{\alpha}
\def\om{\omega}
\def\Om{\Omega}
\def\ga{\gamma}
\def\si{\sigma}

\def\la{\lambda}

\newtheorem{thm}{Theorem}[section]

\newtheorem{lem}[thm]{Lemma}

\theoremstyle{definition}
\newtheorem{defn}[thm]{Definition}
\theoremstyle{remark}

\theoremstyle{example}

\def\C{\mathbb C}

\title{On nonlinear Miyadera-Voigt perturbations}

\author{Mohamed Fkirine}
\address{Department of Mathematics, Faculty of Sciences Hay Dakhla, BP8106, 80000--Agadir, Morocco}
\email{fkirinemohamed@gmail.com}

\author{Said Hadd}
\address{Department of Mathematics, Faculty of Sciences Hay Dakhla, BP8106, 80000--Agadir, Morocco}
\email{s.hadd@uiz.ac.ma}

\subjclass{Primary 60H15, 35F20, 47H14; Secondary 35R60, 35B25, 30K40}

\keywords{Unbounded nonlinear perturbations, Maximal $L^p$-regularity, Semilinear equations}

\begin{document}

\maketitle

\begin{abstract}
  Let $A,C,P:D(A)\subset X\to X$ be  linear operators on a Banach space $X$ such that $-A$ generates a strongly continuous semigroup on $X$, and $F:X\to X$ be a globally Lipschitz function. We study the well-posedness of semilinear equations of the form $\dot{u}(t)=G(u(t))$, where $G:D(A)\to X$ is a nonlinear map defined by $G=-A+C+F\circ P$. In fact, using the concept of maximal $L^p$-regularity and a fixed point theorem, we establish the existence and uniqueness of a strong solution for the above-mentioned semilinear equation. We illustrate our results by applications to nonlinear heat equations with respect to Dirichlet and Neumann boundary conditions, and a nonlocal unbounded nonlinear perturbation.
\end{abstract}


\section{Introduction}\label{sec:1}
Semilinear parabolic equations of the type
\begin{align*}
	\dot{u}(t)+Au(t)=f(u(t)),\quad u(0)=x,\quad t\in [0,\tau],
\end{align*}
where $-A:D(A)\subset X\to X$ is a generator of an analytic semigroup on a Banach space $X$ and $f:X_\al\subset X\to X$ is a locally Lipschitz function with $X_\al$, $\al\in (0,1)$, can be an interpolation space between $D(A)$ and $X$ or $X_\al=D((-A)^\al)$, is well-studied, see e.g. Lunardi \cite{lunardi-95}. As already discussed in \cite[page 254]{lunardi-95}, we may have situations where $f$ is set to a domain that is neither an
interpolation space between $X$ and $D(A)$ nor $X_\al$.

In the present work, we study the well-posedness of the following abstract parabolic equation
\begin{align} \label{main-equa}
	\dot{u}(t)+Au(t)=Pu(t)+F(C u(t)),\quad u(0)=x,\quad t\in [0,\tau],
\end{align}
where $F:X\to X$ is a globally Lipschitz function and $P,C:D(A)\to X$ are linear operators not generally closed or closeable. If we select  \begin{align}\label{non-linear-perturbations}f:D(A)\to X,\quad f(x)=Px+F(Cx),\end{align} then we may find a constant $\ga>0$ such that
\begin{align*}
	\|f(x)-f(y)\|\le \ga \|x-y\|_{D(A)},\qquad (x,y\in D(A)).
\end{align*}
This estimate shows that $f$ is not a Lipschitz function with respect to the norm of $X$. This
fact offers many difficulties in applying the usual fixed point theorems and thus makes the study of the existence and uniqueness of the solutions of the equation \eqref{main-equa}  interesting. It is worth noting that the key references \cite[Chap.11]{CZ-20} and \cite[Chap.3]{lunardi-95} do not cover semilinear parabolic equations of the type  \eqref{main-equa}. We also mention that nonlinear perturbations of the kind \eqref{non-linear-perturbations} appeared for the first time in \cite{TV-03}, where $A$ is a Hille-Yosida operator. On the other hand,  the author of \cite{Web1972} studied operators of the form $-A+B$ with $B:X\to X$ is a continuous nonlinear accretive operator. Some applications of unbounded nonlinear perturbations can also be found in \cite{Web1976}. Recently the work \cite{FH-21} treated the well-posedness of the stochastic version of \eqref{main-equa}, where the approach is based on the concept of Yosida extensions and the approximation theory.

In the present work, we propose another approach based on the concept of maximal $L^p$-regularity to prove the existence and uniqueness of strong solutions of the semilinear equation \eqref{main-equa}. To this end, we assume that $A$ has the maximal $L^p$-regularity ($p\in (1,\infty)$), and for some $\al>0$, there exists $\ga >0$ such that
\begin{align}\label{Admi}
	\|Pe^{-\cdot A}x\|_{L^p([0,\al],X)} + \|Ce^{-\cdot A}x\|_{L^p([0,\al],X)} \le \ga \|x\|\qquad\quad (x\in D(A)).
\end{align}
Further, we assume that the nonlinear function $F:X\to X$ satisfies
\begin{align}\label{Lip}
	\|F(x)-F(y)\|\le \kappa \|x-y\|,\qquad (x,y\in X)
\end{align}
for a constant $\kappa>0$. We mention that operators satisfying the condition \eqref{Admi} are called admissible observation operators, see e.g. \cite[Chap.3]{TW-09} for more properties on such operators. Furthermore, using the conditions \eqref{Admi}-\eqref{Lip} and  H\"older's inequality, we can prove that there exist $\alpha_0>0$ and $\tilde{\ga}\in (0,1)$ such that
\begin{align}\label{N-MV}
\int^{\alpha_0}_0 \| f(T(t)x)-f(T(t)y)\|dt\le \tilde{\ga} \|x-y\|
\end{align}
for any $x,y\in D(A)$. A map $f$ satisfying the estimate \eqref{N-MV} will be called a  nonlinear Miyadera-Voigt perturbation for $A$ (see \cite{TV-03},  \cite{Web1976}).

By assuming that conditions \eqref{Admi}-\eqref{Lip} and  $A$ has the  maximal $L^p$-regularity, we will prove that the semilinear equation \eqref{main-equa} admits a unique strong solution $u\in W^{1,p}([0,\tau],X)\cap L^p([0,\tau],D(A))$ whenever the initial state $x$ belongs to the trace space.

In section \ref{sec:2}, we first give a concise background on maximal regularity and prove Theorem \ref{main-1}, the main result of the paper. In the last section, we illustrate our results by applications to nonlinear heat equation with respect to Dirichlet and Neumann boundary conditions.

\section{Existence, uniqueness and regularity of solution of a class of semilinear parabolic equations}\label{sec:2}
Let $A:D(A)\subset X\to X$ be a linear closed and densely defined operator on a Banach space $X$. Let $p>1$ and $\tau>0$ be real numbers and take  $g\in L^p([0,\tau],X)$. A strong solution of the equation
\begin{align}\label{para}
	\dot{u}(t)+Au(t)=g(t),\quad u(0)=0,\quad t\in [0,\tau],
\end{align}
is a function $u\in W^{1,p}([0,\tau],X)\cap L^p([0,\tau],D(A))$ such that $u(0)=0$ and $u$ satisfies \eqref{para} for almost every $t\in [0,\tau]$.

\begin{defn}\label{MR} We say that $A$ has the maximal $L^p$-regularity on $[0,\tau]$ if for any $g\in L^p([0,\tau],X)$ there exists a unique strong solution  $u\in W^{1,p}([0,\tau],X)\cap L^p([0,\tau],D(A))$ of the problem \eqref{para}.
\end{defn}
It is to be noted that the concept of maximal $L^p$-regularity is independent of $\tau>0$ and $p\in (0,\infty)$. We then only talk about maximal $L^p$-regularity. Moreover if $A$ has maximal $L^p$-regularity then the operator $(-A,D(A))$ generates an analytic semigroup on $X$. The converse is also true if $X$ is a Hilbert space. For more details and justifications of these facts, we refer to e.g. \cite{Dore}.

Motivated by the Definition \ref{MR}, we define the space of maximal regularity
\begin{align*}
	MR_p(0,\tau)&:= W^{1,p}([0,\tau],X)\cap L^p([0,\tau],D(A)).
\end{align*}
On this space we define the following norm
\begin{align*}
 \|u\|_{MR_p}&:=\|u\|_{ W^{1,p}([0,\tau],X)}+\|u\|_{L^p([0,\tau],D(A))},\qquad u\in MR_p(0,\tau).
\end{align*}
Then $(MR_p(0,\tau),\|\cdot\|_{MR_p})$ is a Banach space. We also define the trace space
\begin{align*}
	Tr_p:=\{u(0): u\in MR_p(0,1)\}.
\end{align*}
Endowed with the following norm
\begin{align*}
	\|x\|_{Tr_p}:=\inf\left\{\|u\|_{ MR_p(0,1)}: u\in  MR_p(0,1),\quad\text{and}\;u(0)=x\right\},
\end{align*}
$Tr_p$ is a Banach space satisfying the following dense and continuous embedding
\begin{align*}
	D(A)\hookrightarrow Tr_p\hookrightarrow X.
\end{align*}
We define the following linear operators
\begin{align*}
	& \mathscr{A} u:= A u(\cdot),\qquad D(\mathscr{A}):=L^p([0,\tau],D(A)),\cr
	& \partial u:=\dot{u},\qquad D(\partial):=\{u\in W^{1,p}([0,\tau]):u(0)=0\}=:W^{1,p}_0([0,\tau],X),\cr & \mathscr{L}_A : =\partial+\mathscr{A},\qquad D(\mathscr{L}_A):=D(\mathscr{A}) \cap D(\partial).
\end{align*}
Observe that \begin{align*} D(\mathscr{L}_A)=\{u\in MR_p(0,\tau):u(0)=0\}:=M^0_p(0,\tau).\end{align*}
The operator $\mathscr{L}_A$ generates an evolution semigroup on $L^p([0,\tau],X)$ given by
\begin{align*}
	(e^{-s \mathscr{L}_A}g)(t)=\begin{cases}e^{-s A}g(t-s),& t\in [s,\tau],\cr 0,& t\in [0,s],\end{cases}
\end{align*}
see \cite[Chap.2]{CL-99}. We also mention that the operator $\mathscr{L}_A$ is used in \cite{Arendt-all}, \cite{Clem} and \cite{DG} to study regularities of evolution equations. If $A$ has the maximal $L^p$-regularity, then the operator $\mathscr{L}_A$ is invertible and the solution of \eqref{para} $u\in M^0_p(0,\tau)$ is given by
\begin{align*}
	u(t)=(\mathscr{L}_A^{-1} g)(t)=\int^t_0 e^{-(t-s)A}g(s)ds,\quad t \in [0,\tau].
\end{align*}
The main result of the paper is the following:
\begin{thm}\label{main-1}
	Assume that $A$ has the maximal $L^p$-regularity and the operators $C,P$ and $F$ satisfy the conditions \eqref{Admi}--\eqref{Lip}. Then for any initial condition $u(0)=x\in Tr_p,$ the parabolic problem \eqref{main-equa} has a unique strong solution $u\in MR_p(0,\tau)$, satisfying
	\begin{align}\label{representation1}
		u(t)=e^{-t A}x+\int^t_0 e^{-(t-s) A}Pu(s)ds+\int^t_0 e^{-(t-s) A}F(Cu(s))ds,\quad t\in [0,\tau].
	\end{align}
\end{thm}
\begin{proof}
	Let $x\in Tr_p$ and $\al\in (0,\tau)$ and select
	\begin{align*}
		\Phi(v)=P(\mathscr{L}^{-1}_A v+e^{-\cdot A}x)+F\left(C(\mathscr{L}^{-1}_A v+e^{-\cdot A}x)\right),\qquad v\in L^p([0,\al],X).
	\end{align*}
	According to \cite[Proposition 3.3]{Hadd-SF}, the condition \eqref{Admi} implies
	\begin{align*}
		&\|P(\mathscr{L}^{-1}_A v+e^{-\cdot A}x)\|_{L^p([0,\al],X)}\le  c_p\left(\al^{\frac{1}{q}} \ga \|v\|_{L^p([0,\tau],X)}+ \ga \|x\|\right),\cr
		&\|C(\mathscr{L}^{-1}_A v+e^{-\cdot A}x)\|_{L^p([0,\al],X)}\le  c_p \left(\al^{\frac{1}{q}} \ga \|v\|_{L^p([0,\tau],X)}+ \ga \|x\|\right)
	\end{align*}
	for a constant $c_p>0$ and $q>1$ with $\frac{1}{p}+\frac{1}{p}=1$. This shows that $\Phi:L^p([0,\al],X)\to L^p([0,\al],X)$. On the other hand, for $v_1,v_2\in L^p([0,\al],X)$, we have
	\begin{align*}
		&\|\Phi(v_1)-\Phi(v_2)\|_{L^p([0,\al],X)}\cr &\qquad \le \|P(\mathscr{L}^{-1}_A (v_1-v_2)\|_{L^p([0,\al],X)}+\kappa \|C(\mathscr{L}^{-1}_A (v_1-v_2)\|_{L^p([0,\al],X)}\cr &\qquad \le (1+\kappa) \ga \al^{\frac{1}{q}} \|v_1-v_2\|_{L^p([0,\tau],X)}.
	\end{align*}
	Choose $\al_0>0$ such that $(1+\kappa) \ga \al^{\frac{1}{q}}_0<1$. Then $\Phi$ is a contraction on $L^p([0,\al_0],X)$. Thus, by using Banach's fixed point theorem, there exists a unique $v\in L^p([0,\al_0],X)$ such that $v=\Phi(v)$. Furthermore, there exists a unique $w\in MR^0_p(0,\alpha_0)$ such that $\mathscr{L}_A w=v$. This shows that $w(0)=0$ and
$$\dot{w}+Aw=\mathscr{L}_A w=P(w+e^{-\cdot A}x)+F(C(w+e^{-\cdot A}x))\quad\text{on}\;[0,\al_0].$$
As $x\in Tr_p,$ then $t\mapsto e^{-t A}x$ is differentiable and $\frac{d}{dt}e^{-t A}x=-Ae^{-t A}x$. Now we put $u(t)=w(t)+e^{-t A}x$ for $t\in [0,\al_0]$. Then $u(0)=x$ and $u\in MR_p(0,\al_0)$. Moreover, by a simple computation, we have $\dot{u}(t)+Au(t)=Pu(t)+F(Cu(t))$ for a.e. $t\in [0,\al_0]$. This shows that $u\in MR_p(0,\al_0)$ is a strong solution of \eqref{main-equa}. Using standard arguments in nonlinear analysis \cite{Paz-83}, one can extend this solution to $[0,\tau]$.
\end{proof}
\section{Dirichlet and Neumann boundary problem for semilinear parabolic equations}

\subsection{Dirichlet  boundary problem}\label{sub1}

Let $\Omega\subset \mathbb{R}^n$ be an open bounded set and put $X=L^2(\Omega)$. We consider the following semilinear initial value problem
\begin{align}\label{exx}
	\begin{cases}
		\partial_t u=(\Delta+(-\Delta)^{\al})u+f\left((-\Delta)^{\beta}u\right) & \text{on} \;[0,T]\times \Om,\\
		u_{|{t=0}}=\xi, & \text{on} \; \Omega,\\ u(t,\cdot)=\mathscr{M}u(t), & \text{on} \; \partial\Om,
	\end{cases}
\end{align}
where $\al, \beta\in (0,\frac{1}{4})$, $\xi\in L^2(\Om)$,
\begin{align}\label{bigM}
		\left(\mathscr{M}\varphi\right)(x)=\int_{\Om}K(y,x)\varphi(y)dy,\quad x\in \partial\Om,\quad \varphi\in L^2(\Om),
	\end{align}
for a kernel $K\in L^\infty(\partial\Om\times \Om)$. We assume that $f$ is a real valued function, globally Lipschitz.  Define the following operators
\begin{align*}
	& A:=-\Delta,\quad D(A)=\left\{\varphi\in H^2(\Om): \varphi(x)=\displaystyle\int_{\Omega}K(x,y)\varphi(y)dy,\; x\in \partial\Om\right\},\cr & P:=(-\Delta)^{\al},\quad D(P):= D(A),\cr
	&  C:=(-\Delta)^{\beta},\quad D(C):= D(A),\cr & (F\phi)(x)=f(\phi(x)), \quad x\in \Omega.
\end{align*}
With the use of the above operators the equation \eqref{exx} can be reformulated as the abstract semilinear equation \eqref{main-equa}. Now according to Theorem \ref{main-1}, to prove that the equation \eqref{exx} admits a unique strong solution, it suffices to show that $A$ has the maximal $L^2$-regularity, $P$ and $C$ satisfy the condition \eqref{Admi}.
\begin{lem}\label{lem1}
	The operator $A$ has the maximal $L^p$-regularity for any $p\in (1,\infty)$.
\end{lem}
\begin{proof}
	As $X$ is a Hilbert space, it suffices to prove that $(-A,D(A))$ generates an analytic semigroup on $X$. We first remark that  the operator defined by $-A_0=\Delta$ with domain $D(A_0)=H^2(\Omega)\cap H^1_0(\Omega)$ is a generator of an analytic semigroup on $L^2(\Om)$. Now let $\mathscr{D}:L^2(\partial\Om)\to L^2(\Om)$ be the Dirichlet map, defined by $y=\mathscr{D} v,$ where $\Delta y=0$ on $\Om$ and $y_{|\partial\Om}=v$. If we put $\theta_\varepsilon:=\frac{1}{4}-\varepsilon$ for $\varepsilon\in (0,\frac{1}{4})$, then $\mathscr{D}\in \mathcal{L}(L^2(\partial\Om),D(A_0^{\theta_\varepsilon}))$, due to \cite{LT-83}, \cite{Was-79}. Now, by the closed graph theorem, it follows that $A_0^{\theta_\varepsilon}\mathscr{D}\in\mathcal{L}(L^2(\partial\Om),L^2(\Om))$.   Then there exists a constant $c>0$ such that
	\begin{align}\label{nice-estim}
\begin{split}
		\|A_0 e^{-tA_0}\mathscr{D}\|&=\|A_0^{1-\theta_\varepsilon} e^{-tA_0}A_0^{\theta_\varepsilon} \mathscr{D}\|\cr & \le c t^{\theta_\varepsilon-1},\quad t>0.
\end{split}
	\end{align}
Let $H^{-1}(\Om)$ be is the topological dual of $H^1_0(\Om)$ with respect to the pivot space $L^2(\Om)$. It is known (see e.g. \cite[page 126]{EN-00}) that the extension of the semigroup $(e^{tA_0})_{t\ge 0}$ to $H^{-1}(\Om)$ is a strongly continuous semigroup $(e^{t\tilde{A}})_{t\ge 0}$ on $H^{-1}(\Om)$ whose generator $\tilde{A}:L^2(\Om)\to H^{-1}(\Om)$ is the extension of $-A_0$ to $L^2(\Om)$. As $\mathscr{D}\varphi\in \ker(\Delta)$, then for any $\varphi\in L^2(\partial\Om),$
\begin{align}\label{G0}
(\Delta- \tilde{A})\mathscr{D}\varphi=(-\tilde{A})\mathscr{D}\varphi:=B\varphi.
\end{align}
If we consider the operator $\mathscr{G} \varphi= \varphi_{|\partial\Om}$, then $\mathscr{D}=\left(\mathscr{G}_{|\ker(\Delta)}\right)^{-1}$.  Then the equation \eqref{G0} becomes
\begin{align*}
\Delta= \tilde{A}+B\mathscr{G}\quad\text{on}\; H^2(\Om).
\end{align*}
Using this relation and the proof  \cite[Theorem 4.1]{HMR-15}, we can write
\begin{align}\label{sum}
-A=\tilde{A}+\mathscr{B}\quad D(A)=\{ \varphi\in L^2(\Om):(\tilde{A}+\mathscr{B})\varphi\in L^2(\Om)\},
\end{align}
where $\mathscr{B}:=B\mathscr{M}\in \mathcal{L}(L^2(\Om), H^{-1}(\Om))$, due to $\mathscr{M}\in\mathcal{L}(L^2(\Om),L^2(\partial\Om))$. On the other hand, for $t>0,$  $p>\frac{1}{\theta_\varepsilon}$ and $\psi\in L^p([0,t],L^2(\Om))$,
\begin{align*}
\int^t_0 e^{(t-s) \tilde{A}}\mathscr{B}\psi(s)ds=\int^t_0 A_0 e^{-(t-s)A_0}\mathscr{D}\mathscr{M}\psi(s)ds
\end{align*}
and, by \eqref{nice-estim} and H\"older's inequality,
\begin{align}\label{Estim2}
\begin{split}
		& \left\|\int^{t}_0 e^{(t-s)\tilde{A}}\mathscr{B} \psi(s)ds\right\|_{L^2(\Om)} \cr & \hspace{1cm} \le c \left(\int^t_0 s^{q(\theta_\varepsilon-1)}ds\right)^\frac{1}{q} \left(\int^t_0 \|\mathscr{M}\psi(s)\|^pds\right)^{\frac{1}{p}}\cr & \hspace{1cm}
\le \tilde{c} \|\mathscr{M}\|\|\psi\|_{L^p([0,t],L^2(\Om))},
\end{split}
	\end{align}
where $\frac{1}{p}+\frac{1}{q}=1,$  $0<q(1-\theta_\varepsilon)<1$ and a constant $\tilde{c}:=\tilde{c}(t,q,\varepsilon,c)>0$. According to \cite[page 188, Corollary 3.4]{EN-00} $\mathscr{B}$ is a Desch-Schappacher perturbation for $-A$. Now the relation \eqref{nice-estim} shows that the operator $-A$ generates a strongly continuous semigroup on $X$ given by
\begin{align}\label{JJ}
e^{-tA}\phi=e^{-tA_0}\phi+\int^{t}_0 e^{-(t-s)A_0}\mathscr{B} e^{-sA} \phi ds
\end{align}
for any $t\ge 0,$ and $\phi\in X$, see e.g. \cite[page 186]{EN-00}. From \eqref{sum},  $\la+A=(\la+A_0)(I-R(\la,-A_0)\mathscr{B})$ for $\la\in\rho(-A_0)$. Thus for any $\la\in\rho(-A_0),$ $\la\in \rho(-A)$ if and only if $1\in \rho(R(\la,-A_0)\mathscr{B})$, and in this case we have
\begin{align*}
R(\la,-A)=(I-R(\la,-A_0)\mathscr{B})^{-1}R(\la,-A_0).
\end{align*}
To be more self-contained, let us recall from the proof of \cite[Theorem 8]{ABDH-20} how to prove that the semigroup generated by $-A$ is analytic. As $-A_0$ generates an analytic semigroup, then there exists $\om>\om_0(-A_0)$ and a constant $c>0$ such that $\|R(\la,-A_0)\|\le c |\la-\om|^{-1}$ for any $\la\in\C$ with ${\rm Re}\la>\om$. On the other hand, from \eqref{Estim2} and \cite[Chap. 3]{TW-09}, we also have $\|R(\la,-A_0)\mathscr{B}\|\le \eta \|\mathscr{M}\| ({\rm Re}\la-\om)^{\frac{-1}{q}}$ for any ${\rm Re}\la>\om$ and for a certain constant $\eta>0$. Now for any $\la\in \C$ such that ${\rm Re}\la>\om_1:=\om+(2 \eta \|\mathscr{M}\|)^{q}$, we have $\|(I-R(\la,-A_0)\mathscr{B})^{-1}\|\le 2,$ and then $\|R(\la,-A)\|\le 2 c |\la-\om_1|^{-1}$. This ends the proof.
\end{proof}
\begin{lem}\label{lem2}
Let $p\in (1,\infty)$	 and  $\si\in(0,\frac{1}{p})$. Then the operator $A^\si$ satisfies the condition \eqref{Admi} with respect to the exponent $p$.
\end{lem}
\begin{proof}
According to the proof of Lemma \ref{lem1}, we know that $-A$ generates an analytic semigroup on $L^2(\Om)$. Thus for any $\si,$ we have
\begin{align*}
\|A^\si e^{-t A}\|\le \frac{M}{t^\si}\qquad (t>0),
\end{align*}
for a constant $M>0$. Now for $\si\in(0,\frac{1}{p})$  the function $t\mapsto \frac{M}{t^\si}$ is $p$-integrable on any interval $[0,r]$ with $r>0,$ and
\begin{align*}
\int^r_0 \|Ae^{-t A} \varphi\|^p\le \ga^p \|\phi\|_{L^2(\Om)}
\end{align*}
for any $\phi\in D(A)$ and a constant $\ga:=\ga(r,p,\si,M)>0$.
\end{proof}
The following result follows immediately from Lemma \ref{lem1},  Lemma \ref{lem2} and Theorem \ref{main-1}.
\begin{thm}
Let $p\in (1,\infty)$ and  $\al,\beta\in(0,\frac{1}{p})$.  Then the semilinear equation \eqref{exx} has a unique strong solution
\begin{align*}
u\in W^{1,p}([0,T],L^2(\Om))\cap L^p([0,T],D(A)).
\end{align*}
\end{thm}
\subsection{Neumann boundary problem with nonlocal perturbation term}\label{sub2}
Let $\Omega\subset \mathbb{R}^n$ be an open bounded set with a $C^2$ boundary $\partial\Om$ and outer unit normal $\nu$ and put $X=L^2(\Omega)$. We consider the following semilinear initial value problem
\begin{align}\label{exam2}
	\begin{cases}
		\dot{u}(t,x)=\Delta u(t,x)+f\left(\displaystyle\int_{\partial\Om} \Upsilon(x,y)c(y)u(t,y)dy\right) & x\in \Omega, t\in[0,T],\\
		u(0,x)=g(x), & x\in \Omega, \xi\in L^2(\Omega),\\  \nabla u(t,x)|\nu(x)=\displaystyle\int_{\Omega}K(x,y)u(t,y)dy, & x\in \partial \Omega, t\in[0,T],
	\end{cases}
\end{align}
where $\Upsilon,K\in L^\infty(\partial \Om\times \Om)$, $c\in C_b(\partial\Om)$, and $f$ is a globally Lipschitz real valued function.  We define the following operators:
\begin{align*}
	&\mathscr{R}:L^2(\partial\Om)\to L^2(\Om),\quad  (\mathscr{R}\varphi)(x)= \int_{\partial\Om} \Upsilon(x,y)\varphi(y)dy,\qquad x\in\Om,\cr
	& (\Theta\varphi)(y)=c(y)\varphi(y),\qquad y\in\partial\Om,\cr & C= \mathscr{R} \Theta: H^2(\Om)\to L^2(\Om).
\end{align*}
On the other hand, let $F:L^2(\Om)\to L^2(\Om)$ and $\mathscr{M}\in\mathcal{L}(L^2(\Om),L^2(\partial\Om))$ as in Subsection \ref{sub1}. We now define the operator
\begin{align*}
	\mathcal{A}:=-\Delta,\quad D(\mathcal{A})=\left\{\varphi\in H^2(\Om): \nabla \varphi|\nu=\mathscr{M}\varphi\;\text{on}\;\partial\Om\right\}.
\end{align*}
Thus the equation \eqref{exam2} is reformulated in $L^2(\Om)$ as
\begin{align*}\label{Neuman-abstrait}
	\dot{u}(t)+\mathcal{A} u(t)=F(Cu(t)),\quad u(0)=g,\quad t\in [0,\tau].
\end{align*}
\begin{thm}\label{main-3}
	The equation \eqref{exam2} has a unique strong solution
\begin{align*}
u\in W^{1,2}([0,\tau],L^2(\Om))\cap L^2([0,\tau],D(A)).
\end{align*}
\end{thm}
\begin{proof}
	Step 1: We will prove that the operator $\mathcal{A}$ has the maximal $L^2$-regularity. In fact, let us first define the operator $\mathcal{A}_0:=-\Delta$ with domain $D(\mathcal{A}_0)=\{\varphi\in H^2(\Om):\nabla \varphi|\nu=0,\;x\in\partial\Om\}$. It is well known that $-\mathcal{A}_0$ generates an analytic semigroup on $L^2(\Om)$. Second, denote by $\varphi=\mathscr{N}\psi\in H^{\frac{3}{2}}(\Om)$ the solution of the elliptic boundary value problem $ \varphi=0$ on $\Om$ and  $\nabla \varphi|\nu=\psi$ on $\partial\Om$ for $\psi\in L^2(\partial\Om)$. Then $\mathscr{N}$ is continuous from $L^2(\partial\Om)$ to $D(\mathcal{A}_0^\beta)$ for any $\beta\in (0,\frac{3}{4})$. Moreover, for any $t>0$, $\al\ge 0$ and $0<\beta<3/4$, we have
	\begin{equation}\label{hadd}
		\mathcal{A}_0^\al e^{-t\mathcal{A}_0}\mathcal{A}_0\mathscr{N}\in \mathcal{L}(L^2(\partial\Om),X)\;\text{and}\; \|\mathcal{A}_0^\al e^{-t\mathcal{A}_0}\mathcal{A}_0\mathscr{N}\| \le \kappa_\beta t^{\beta-\al-1}
	\end{equation}
	for a constant $\kappa_\beta>0,$ due to \cite{LT}. We put $\mathscr{P}:=\mathcal{A}_0 \mathscr{N}\mathscr{M}$ and choose $\beta\in (\frac{1}{2},\frac{3}{4})$ and $\al=0$. Then by using \eqref{hadd} and the fact that  $2(1-\beta)<1$, we obtain
	\begin{align*}
		\left\|\int^t_0 e^{-(t-s)\mathcal{A}_0}\mathscr{P}v(s)ds\right\|_{L^2(\Om)}\le \delta_\beta \|\mathscr{M}\|\|v\|_{L^2([0,t],L^2(\Om))}
	\end{align*}
	for all $t>0$ and $v\in L^2([0,t],L^2(\Om)),$ where $\delta_\beta>0$ is a constant. By the same argument as in Subsection \ref{sub1}, $\mathscr{P}$ is a Desch-Schappacher perturbation operator for $\mathcal{A}_0$ which implies that the operator $\mathfrak{A}:=-\mathcal{A}_0+\mathscr{P}$ with domain $D(\mathfrak{A})=\{\varphi\in L^2(\Om):(\mathcal{A}_0+\mathscr{P})\varphi\in L^2(\Om)\}$ generates an analytic semigroup on $L^2(\Om)$. Furthermore, we have $-\mathcal{A}=\mathfrak{A},$ see e.g. \cite{HMR-15}. Thus $-\mathcal{A}$ has the maximal $L^2$-regularity. In addition,
	\begin{align}\label{E1}
		e^{-t \mathcal{A}}\varphi=e^{-t \mathcal{A}_0}\varphi+\int^t_0 e^{-(t-s) \mathcal{A}_0}\mathscr{P}e^{-s \mathcal{A}}\varphi ds
	\end{align}
	for any $t>0$ and $\varphi\in L^2(\Om)$.\\
	Step 2: We will prove that the operator $C$ satisfies \eqref{Admi} with respect to $\mathcal{A}$. In fact, the operator $C:D(\mathcal{A}_0^\al)\to L^2(\Om)$ is uniformly bounded for any $\al>\frac{1}{4}$, so that $C\mathcal{A}_0^{-\al}\in \mathcal{L}(L^2(\Om))$ and $\|C\mathcal{A}_0^{-\al}\|\le \eta$ for a constant $\eta>0$.  Let us now choose $\al\in (\frac{1}{4},\frac{1}{2})$. We have
	\begin{align}\label{E2}
		\|Ce^{-t\mathcal{A}_0}\|\le \eta \frac{M}{t^{\al}},\quad (t>0).
	\end{align}
	On the other hand, by using \eqref{hadd} and the fact that then $\frac{1}{2}<1-(\beta-\alpha)<1$, we obtain
	\begin{align}\label{E3}
		\left\|C\int^t_0 e^{-(t-s) \mathcal{A}_0}\mathscr{P}e^{-s \mathcal{A}}\varphi ds\right\|& \le \eta \kappa_\beta \int^t_0 \frac{1}{(t-s)^{1-(\beta-\alpha)}} \|\mathscr{M}e^{-s\mathcal{A}}\varphi\|ds\cr & \le \eta \kappa_\beta \|\mathscr{M}\| e^{|\om|t} t^{\beta-\al} \|\varphi\|
	\end{align}
	for any $\varphi\in D(\mathcal{A}),$ and $\om>\om_0(-\mathcal{A}),$ the type of the semigroup generated by $-\mathcal{A}$. Now the fact that $C$ satisfies \eqref{Admi} for $\mathcal{A}$ follows by combining \eqref{E1}, \eqref{E2} and \eqref{E3}. Finally, we can use Theorem \ref{main-1} to conclude.
	
\end{proof}

\bibliography{bibliography}



\end{document}